\documentclass[12pt]{article}
\usepackage{amsmath}
\usepackage{amsfonts}
\usepackage{amsthm}
\usepackage{graphicx}
\usepackage{overpic}
\usepackage{a4wide}
\usepackage{amssymb} 
\usepackage{pictex}
\usepackage{rotating}  
\usepackage{color}
\usepackage{etex} 
\usepackage{enumitem}
\usepackage{accents}

\definecolor{refkey}{rgb}{0,0,1}
\definecolor{labelkey}{rgb}{1,0,0}

\usepackage{comment}

\numberwithin{equation}{section}

\newtheorem{theorem}{Theorem}[section]
\newtheorem{proposition}[theorem]{Proposition}

\newtheorem{Definition}[theorem]{Definition}

\newtheorem{Remark}[theorem]{Remark}
\newenvironment{remark}{\begin{Remark}\rm}{\end{Remark}}
\newtheorem{Example}[theorem]{Example}
\newenvironment{example}{\begin{Example}\rm}{\end{Example}}
\newtheorem{RHproblem}[theorem]{RH problem}

\usepackage{color}

\newcommand{\C}{\mathbb{C}}

\newcommand{\R}{\mathbb{R}}

\renewcommand{\Re}{{\rm Re} \,}
\renewcommand{\Im}{{\rm Im} \,}

\renewcommand{\hat}{\widehat}

\def\det{\mathop{\mathrm{det}}\nolimits}

\usepackage[bookmarksopen, naturalnames]{hyperref}

\begin{document}
\title{Bernstein-Markov: a survey}
\author{Thomas Bloom, Norman Levenberg, Federico Piazzon and Franck Wielonsky}
\maketitle 

\begin{abstract} We give a survey of recent results, due mainly to the authors, concerning Bernstein-Markov type inequalities and connections with potential theory.

\end{abstract}

\section{Introduction}  Given a compact set $K\subset \C^n$, a finite measure $\nu$ on $K$ is a Bernstein-Markov measure if there is a strong comparability between $L^{\infty}$ and $L^2$ norms on $K$ for polynomials of a given degree. Precisely, let $\mathcal P_k$ denote the (complex) vector space of (holomorphic) polynomials $p_k(z_1,...,z_n)$ of degree at most $k$ in $\C^n$. We say that $(K,\nu)$ satisfies a {\it Bernstein-Markov property} if for all $p_k\in \mathcal P_k$, 
$$||p_k||_K:=\sup_{z\in K} |p_k(z)|\leq  M_k||p_k||_{L^2(\nu)}  \ \hbox{with} \ \limsup_{k\to \infty} M_k^{1/k} =1.$$
It is often useful to state this as: given $\epsilon >0$, there exists $C=C(\epsilon,K)$ such that for all $k=1,2,...$
\begin{equation}\label{used} ||p_k||_K\leq  C(1+\epsilon)^k||p_k||_{L^2(\nu)}\end{equation}
for all $p_k\in \mathcal P_k$. Equivalently, for any sequence $\{p_j\}$ of nonzero polynomials with $deg(p_j)\to +\infty$,
\begin{equation}\label{bmk} \limsup_{j\to \infty} \bigg(\frac{||p_j||_K}{||p_j||_{L^2(\mu)}}\bigg)^{1/deg (p_j)} \leq 1.\end{equation}
More generally, for $K\subset \C^n$ compact, $Q:K\to \R$ lowersemicontinuous, and $\nu$ a finite measure on $K$, we say that the triple $(K,\nu,Q)$ satisfies a {\it weighted Bernstein-Markov property} if for all $p_k\in \mathcal P_k$, 
$$||e^{-kQ}p_k||_K \leq M_k ||e^{-kQ}p_k||_{L^2(\nu)} \ \hbox{with} \ \limsup_{k\to \infty} M_k^{1/k} =1.$$
Setting $w:=e^{-Q}$ we rewrite the previous inequality as
$$||p_kw^k||_K \leq M_k ||p_kw^k||_{L^2(\nu)}.$$
This is a strong comparability between $L^{\infty}$ and $L^2$ norms on $K$ for {\it weighted} polynomials of a given degree. 
Equivalently, for any sequence $\{p_j\}$ of nonzero polynomials with $deg(p_j)\to +\infty$,
\begin{equation}\label{bmwk}\limsup_{j\to \infty} \bigg(\frac{||w^{deg (p_j)} p_j||_K}{||w^{deg (p_j)}p_j||_{L^2(\mu)}}\bigg)^{1/deg (p_j)} \leq 1.\end{equation}

An elementary example is (normalized) arclength measure $d\nu = \frac{1}{2\pi} d\theta$ on the unit circle $K=\{z\in \C:|z|=1\}$ in the complex plane. Here, the monomials $\{b_j(z)=z^j\}_{j=0}^k$ form an orthonormal basis for $\mathcal P_k$ in $L^2(\nu)$ and if we form the {\it $k-$th Bergman function}
$$B_k^{\nu}(z):=\sum_{j=0}^k |b_j(z)|^2=\sum_{j=0}^k |z^j|^2,$$
for $\nu$, then 
\begin{enumerate}
\item writing $p_k\in \mathcal P_k$ as $p_k(z)=\sum_{j=0}^k c_j b_j(z)$, we have 
$$|p_k(z)|\leq \sqrt {B_k(z)}\cdot ||p_k||_{L^2(\nu)};$$
\item $\max_{z\in K}B_k^{\nu}(z) = k+1$.
\end{enumerate}
Then 1. and 2. show that 
\begin{equation} \label{circleest} ||p_k||_K\leq (k+1)^{1/2} ||p_k||_{L^2(\nu)}.\end{equation}
Indeed, it will follow from results in section 3 that for $K=\{z\in \C:|z|=1\}$, if $Q$ is continuous on $K$ then $(K,\nu,Q)$ satisfies a weighted Bernstein-Markov property. Clearly one can replace $K$ by the closed unit disk $K'=\{z: |z| \leq 1 \}$ and (\ref{circleest}) holds. However, with this new set $K'$, and $d\nu = \frac{1}{2\pi} d\theta$ on the boundary, the function $Q(z) = |z|^2$ gives an example where $(K',\nu,Q)$ does not satisfy a weighted Bernstein-Markov property. This is easily checked by comparing supremum and $L^2(\nu)$ norms of the weighted monomials $e^{-k|z|^2}z^k$ on $K'$.

These Bernstein-Markov properties and their generalizations have many applications in multivariate approximation theory and other areas of mathematics. In this survey, we restrict to the properties themselves and the connections with potential theory due mostly to the authors. Applications may be found in Remarks \ref{ldp} and \ref{ldp2} and the indicated references. Unless otherwise noted, {\sl we will consider (weighted) Bernstein-Markov properties only for measures $\mu$ of finite mass on compact sets $K$}. In this setting, one can replace $L^2$ by $L^p$ for any $0<p <\infty$ in (\ref{bmk}), (\ref{bmwk}); cf., the arguments in Theorem 3.4.3 of \cite{ST}. We will generally restrict to $p=1$ or $p=2$; the latter case allows one connections with other notions (cf., \cite{optimal}). In this survey, we are not concerned in explicitly evaluating the Bernstein-Markov constants $\{M_{k}\}$. Inequalities between different $L^{p}$ norms, $0<p\leq\infty$, and with explicit constants $\{M_{k}\}$ are usually called {\it Nikolskii inequalties} in the literature. They are usually established on specific subsets of $\C$; e.g., subsets of $\R$ or the unit circle. We refer the reader to Chapter VI of \cite{ST} for more information and further references.

When possible, we give results valid in $\C^n$ for $n\geq 1$. Much of the univariate theory can be found in \cite{ST} and \cite{SaTo}. The next section provides the requisite background in potential theory in $\C$ and pluripotential theory in $\C^n, \ n>1$. Section 3 gives some sufficient conditions on a measure $\mu$ to satisfy a Bernstein-Markov property as in (\ref{bmk}) and (\ref{bmwk}). In the ensuing two sections we work in $\C$: rational Bernstein-Markov properties are the topic of section 4 while section 5 introduces new extremal-like functions and Bernstein-Walsh estimates in the univariate setting. Section 6 contains some estimates related to Riesz potentials in $\R^d$. The final three sections include a discussion of Bernstein-Markov properties on algebraic subvarieties of $\C^n$; an example of V. Totik; a recent result of T. Bloom; and some open questions.

\section{Potential theory and Bernstein-Markov properties}

A key tool to the connection between Bernstein-Markov properties and potential theory (in $\C$) or pluripotential theory (in $\C^n, \ n>1$) is the {\it extremal function} 
\begin{equation}\label{veekpoly}V_K(z)= \sup \{{1\over deg (p)}\log |p(z)|: p \in \bigcup_k \mathcal P_k, \ ||p||_{K}:=\sup_K |p|\leq 1 \}\end{equation}
associated to a compact set $K\subset \C^n$. More generally, for $Q:K\to \R$ lowersemicontinuous, we define the {\it weighted extremal function}
\begin{equation}\label{veekqpoly}V_{K,Q}(z)= \sup \{{1\over deg (p)}\log |p(z)|: p \in \bigcup_k \mathcal P_k, \ ||w^{deg (p)}p||_{K}:=\sup_K |w^{deg (p)}p|\leq 1 \}\end{equation}
where $w=e^{-Q}$. If $V_{K,Q}$ is upper semicontinuous (usc), then for $n=1$ it is a {\it subharmonic} (shm) function and for $n>1$ it is {\it plurisubharmonic} (psh). A real-valued function $u$ on a domain $D\subset \C^n$ is psh in $D$ if it is usc and $u|_{D\cap L}$ is shm on components of $D\cap L$ for each complex line $L\subset \C^n$. We say a set $E\subset \C^n$ is {\it pluripolar} if there exists $u$ psh in a neighborhood of $E$ with $u\not \equiv -\infty$ and $E\subset \{u= -\infty\}$. In one variable, where psh $=$ shm, we say $E$ is a {\it polar} set. If $f:D\subset \C^n\to \C$ is holomorphic, then $u(z)=\log |f(z)|$ is psh in $D$. This observation yields standard examples of pluripolar sets: analytic varieties. For general compact sets $K\subset \C^n$, the usc regularization
$$V^*_{K
}(z):=\limsup_{\zeta \to z}V_{K}(\zeta)$$
is either identically $+\infty$ -- which occurs if $K$ is pluripolar -- or $V_K^*$ is psh. For $K=\{z\in \C^n: |z| =1\}$ and $K'=\{z\in \C^n: |z| \leq 1\}$, 
\begin{equation}\label{veek}V_K(z)=V_{K'}(z)=V^*_K(z)=V^*_{K'}(z)=\log^+|z|:=\max[0,\log |z|].\end{equation}
For the weighted extremal functions, one {\it assumes} that the set $\{z\in K: w(z) > 0\}$
is not pluripolar (and $Q$ is called {\it admissible}); thus a priori $K$ itself is not pluripolar and $V_{K,Q}^*$ is always psh. 

If $n=1$, the Laplacian of $V_K^*$ (if $K$ is not polar) and $V_{K,Q}^*$ have important interpretations. Let $K\subset \C$ be compact and let ${\mathcal M}(K)$ denote the convex set of probability  measures on $K$. One considers the  {\it energy minimization problem}: minimize the {\it logarithmic energy} 
$$I(\mu):=\int_K \int_K \log {1\over |z-\zeta|}d\mu(z) d\mu(\zeta)=\int_K U_{\mu}(z)d\mu(z)$$
over $\mu \in {\mathcal M}(K)$. Here 
$$U_{\mu}(z):=\int_K \log {1\over |z-\zeta|}d\mu(\zeta)$$ is called the {\it logarithmic potential} of $\mu$. Either $\inf_{\mu \in {\mathcal M}(K)} I(\mu)=:I(\mu_{K})<+\infty$ for a unique $\mu_K\in {\mathcal M}(K)$ or else $I(\mu)=+\infty$ for all $\mu \in {\mathcal M}(K)$. In the former case $V_K^*=-U_{\mu_K}+I(\mu_K)$ on $\C$ and $\mu_K = \frac{1}{2\pi} \Delta V_K^*$. In the latter case $K$ is polar. We call 
\begin{equation} \label{cap} C(K):=\exp \bigl[-\inf_{\mu \in {\mathcal M}(K)} I(\mu)\bigr]\end{equation}
the {\it logarithmic capacity} of $K$. One can discretize this problem. For each $k=1,2,...$, given $k+1$ points $z_0,...,z_k$, consider the discrete measure $\mu_k:=\frac{1}{k} \sum_{j=0}^k \delta (z_j)$ where $\delta(z_j)$ denotes the unit measure at the point $z_j$. The logarithmic energy of $\mu_k$ omitting ``diagonal terms'' is 
$$ \frac{2}{k(k+1)}\sum_{j<l} \log \frac{1}{|z_j-z_l|}.$$
Exponentiating the negative of this quantity leads to the maximization problem
\begin{equation}\label{deltak} \delta_k(K):=\max_{z_0,...,z_k\in K}\prod_{j<l}|z_j-z_l|^{1/{k +1\choose 2}},\end{equation}
the $k-th$ order diameter of $K$. Note 
\begin{equation}\label{deltak2} VDM_k(z_0,...,z_k)=\det [z_i^{j}]_{i,j=0,1,...,k}=\prod_{j<l}(z_j-z_l)\end{equation}
$$=\det\left[
\begin{array}{ccccc}
 1 &z_0 &\ldots  &z_0^k\\
  \vdots  & \vdots & \ddots  & \vdots \\
1 &z_k &\ldots  &z_k^k
\end{array}
\right] $$
is the classical Vandermonde determinant where the basis monomials $1,z,...,z^k$ for $\mathcal P_k$ are evaluated at $z_0,...,z_k$. Optimizing points $\lambda_0,...,\lambda_k\in K$ for (\ref{deltak}) are called {\it Fekete points of order $k$}. It is straightforward that the sequence $\{\delta_k(K)\}$ decreases with $k$; the quantity 
$$\delta(K):=\lim_{k\to \infty} \delta_k(K)$$ is called the {\it transfinite diameter} of $K$ and coincides with $C(K)$ in (\ref{cap}). For proofs of the statements in this paragraph, see \cite{R}.
 
Noting that {\sl $z_j\to VDM_k(z_0,...,z_k)$ is a polynomial of degree $k$ in $z_j$}, by repeatedly applying the Bernstein-Markov property we can recover the transfinite diameter $\delta(K)$ in an $L^2-$fashion with a Bernstein-Markov measure.

\begin{proposition} \label{znunwtd} Let $K\subset \C$ be compact and let $(K,\nu)$ satisfy a Bernstein-Markov property. Then
$$\lim_{k\to \infty} Z_k^{1/k^2} = \delta(K)$$ where
$$Z_k=Z_k(K,\nu):= \int_{K^{k+1}} |VDM_k(z_0,...,z_k)|^2d\nu(z_0) \cdots d\nu(z_k).$$
\end{proposition}

\begin{proof} Since $Z_k\leq \delta_k(K)^{k(k+1)}\cdot \nu(K)^{k+1}$, clearly
$$\limsup_{k\to \infty} Z_k^{1/k^2} \leq \delta(K).$$
Choosing $\lambda_0,...,\lambda_k\in K$ to be Fekete points of order $k$, 
$$p_0(z):=VDM_k(z,\lambda_1,...,\lambda_k)$$
is a polynomial of degree $k$ with 
$$||p_0||_K^2= |VDM_k(\lambda_0,...,\lambda_k)|^2=\delta_k(K)^{k(k+1)}.$$
Thus
$$\delta_k(K)^{k(k+1)}\leq M_k^2 \cdot ||p_0||_{L^2(\nu)}^2=M_k^2 \cdot \int_K |VDM_k(z,\lambda_1,...,\lambda_k)|^2d\nu(z).$$
Fixing $z\in K$ and applying the Bernstein-Markov property to $p_1(w):=VDM_k(z,w,\lambda_2,...,\lambda_k)$ (note $|p_1(\lambda_1)|\leq ||p_1||_K$) and continuing, we obtain
$$\delta_k(K)^{k(k+1)}\leq M_k^{2(k+1)}\cdot Z_k$$
so that
$$\delta(K)\leq \liminf_{k\to \infty} Z_k^{1/k^2} .$$
\end{proof}

In the weighted setting, $K\subset \C$ should be closed but may be unbounded. We require $w=e^{-Q}\geq 0$ to be usc on $K$ with
$\{z\in K:w(z)>0\}$ nonpolar; if $K$ is unbounded, we must impose a certain growth property on $w$. We will, however, restrict to the compact case. We minimize the weighted energy 	
\begin{equation}\label{wten} I^Q(\tau):=\int_K \int_K \log \frac {1}{|z-t|w(z)w(t)}d\tau(t)d\tau(z)=I(\tau) +2\int_KQd\tau\end{equation} 
over all $\tau \in \mathcal M(K)$. There is always a unique minimizer which we denote by $\mu_{K,Q}$. Then 
$$\mu_{K,Q}=\frac{1}{2\pi}\Delta V_{K,Q}^*.$$
 The associated discretization leads to the {\it weighted transfinite diameter} of $K$ with respect to $w$:
	\begin{equation} \label{wtdvdm} \delta^Q(K):=\lim_{k\to \infty} \bigl[\max_{\lambda_i\in K}|VDM_k(\lambda_0,...,\lambda_k)|e^{-kQ(\lambda_0)}\cdots e^{-kQ(\lambda_k)}\bigr] ^{1/{k+1\choose 2}}= e^{-I^Q(\mu_{K,Q})}.\end{equation} 
	The proof that the limit exists and equals $e^{-I^Q(\mu_{K,Q})}$ is exactly as in the unweighted case. Maximizing points $\lambda_0,...,\lambda_k\in K$ at the $k-$th stage are called weighted Fekete points of order $k$. For future use we introduce notation for the weighted Vandermonde determinants:
	$$VDM_k^Q(\lambda_0,...,\lambda_k):=VDM_k(\lambda_0,...,\lambda_k)e^{-kQ(\lambda_0)}\cdots e^{-kQ(\lambda_k)}$$
	$$=VDM_k(\lambda_0,...,\lambda_k)w(\lambda_0)^k\cdots w(\lambda_k)^k.$$ 
Noting that $z_j\to VDM_k^Q(z_0,...,z_k)$ is a {\it weighted} polynomial of degree $k$ in $z_j$; i.e., a function of the form $w(z_j)^k\cdot p_k(z_j)$ where $p_k\in \mathcal P_k$, it is easy to see that the analogue of Proposition  \ref{znunwtd} holds in the weighted case; i.e., if  $(K,\nu,Q)$ satisfy a Bernstein-Markov property and 
\begin{equation} \label{zkq}Z_k=Z_k(K,\nu,Q):= \int_{K^{k+1}} |VDM_k^Q(z_0,...,z_k)|^2d\nu(z_0) \cdots d\nu(z_k),\end{equation} 
then $\lim_{k\to \infty} Z_k^{1/k^2} = \delta^Q(K)$.

\begin{remark}\label{ldp} Proposition \ref{znunwtd} and its generalizations are key tools in proving certain probabilistic results involving equidistribution of discrete measures associated to random arrays of points in a compact set $K$ to an equilibrium measure. Let $(K,\nu,Q)$ satisfy a Bernstein-Markov property. With
$Z_k$ as in (\ref{zkq}), we define a probability measure $Prob_k$ on $K^{k+1}$: for a Borel set $A\subset K^{k+1}$,
$$Prob_k(A):=\frac{1}{Z_k}\cdot \int_A  |VDM_k^Q(z_0,...,z_{k
})|^2 \cdot d\nu(z_0) \cdots d\nu(z_{k
}).$$
We get an induced probability measure ${\bf P}$ on the infinite product space of arrays $\chi:=\{X=\{z_j^{(k)}\}_{k=1,2,...; \ j=0,1,...,k}: z_j^{(k)} \in K\}$: 
	$$(\chi,{\bf P}):=\prod_{k=1}^{\infty}(K^{k},Prob_k).$$
Then for ${\bf P}$-a.e. array $X=\{z_j^{(k)}\}\in \chi$, 
\begin{equation}\label{nuk}\nu_k :=\frac{1}{k}\sum_{j=0}^{k}\delta (z_j^{(k)})\to \mu_{K,Q} \ \hbox{weak-*}.\end{equation}
Much more is true; cf., Remark \ref{ldp2}. We refer the reader to \cite{PELD}, \cite{VELD}, \cite{LPLD} and \cite{BLTW} for extensions and generalizations of this argument.
\end{remark}
		
 In $\C^n$ for $n>1$ we have a generalization of (weighted) Vandermonde and (weighted) transfinite diameter. For $k=1,2,...$ we write $N_k:=$ dim${\mathcal P}_k={n+k\choose k}$ and let $\{e_1,...,e_{N_k}\}$ be the standard monomial basis for ${\mathcal P}_k$. Given $K\subset \C^n$ and an admissible weight function $Q$ on $K$, the weighted $k-$th order diameter of $K$ with $Q$ is
$$\delta^{Q,k}(K):= \bigl(\max_{{\bf x}=(x_1,...,x_{N_k})\in K^{N_k}} |VDM_k({\bf x})| e^{-kQ(x_1)} \cdots e^{-kQ(x_{N_k})}\bigr)^{\frac{n+1}{nkN_k}}$$
where $VDM_k({\bf x}):=\det [e_i(x_j) ]_{i,j=1,...,N_k}$ (compare (\ref{deltak2})). Here $\frac{nkN_k}{n+1} = \sum_{j=1}^{N_k}deg(e_j)$. That the limit $$\lim_{k\to \infty} \delta^{Q,k}(K)= \delta^Q (K)$$
exists is a nontrivial result due to Zaharjuta \cite{Z} in the unweighted case (cf., \cite{BBnew} or \cite{BLtd} for the weighted case). The limit is called the weighted transfinite diameter of $K$ with respect to $Q$. For $Q\equiv 0$ we call $\delta (K):=\delta^0 (K)$ the transfinite diameter of $K$. We have the following weighted generalization of Proposition \ref{znunwtd} (with a similar proof), valid in $\C^n$ for $n\geq 1$.

\begin{proposition}
\label{weightedtd}
Let $K\subset \C^n$ be a compact set and let $Q$ be an admissible weight function on $K$. If $\nu$ is a measure on $K$ with $(K,\nu,Q)$ satisfying a weighted Bernstein-Markov property, then 
\begin{equation*}
\lim_{k\to \infty} Z_k^{\frac{1}{2kN_k}
} =   \delta^Q(K)
\end{equation*}
where
$$Z_k=Z_k(K,\nu,Q):= \int_{K^{N_k}}  |VDM_k({\bf x})|^2 e^{-2kQ(x_1)} \cdots e^{-2kQ(x_{N_k})}d\nu({\bf x}).$$\end{proposition}

In this higher dimensional setting, there is no natural notion of energy as the associated ``potential theory'' is nonlinear -- the replacement for the (linear) Laplacian 
$$\Delta = \frac{\partial^2 u}{\partial x^2}+ \frac{\partial^2 u}{\partial y^2} =4\frac{\partial^2 u}{\partial z \partial \bar z}$$ 
in $\C$ is the (nonlinear) complex Monge-Amp\`ere operator $(dd^c \cdot)^n$ in $\C^n$. For a $C^2-$real-valued function $u=u(z_1,...,z_n)$, 
$$(dd^c u)^n =c_n \det [\frac{\partial^2 u}{\partial z_j \partial \bar z_k}]_{j,k=1,...,n}\ dV$$
where $c_n$ is a dimensional constant and $dV$ is Lebesgue measure. For locally bounded psh functions $u$, the complex Monge-Amp\`ere operator $(dd^c u)^n$ is well-defined as a positive measure (cf., \cite{K}). With $V_{K,Q}$ in (\ref{veekqpoly}), the replacement for the weighted energy minimizing measure in this situation is given by
$$\mu_{K,Q}:= (dd^c V_{K,Q}^*)^n.$$
The constant $c_n$ is chosen to make this a probability measure.

\section{Mass density conditions for Bernstein-Markov}

Given a compact set $K\subset \C^n$, which measures $\mu$ with support in $K$ satisfy a Bernstein-Markov property? Is there a ``checkable'' sufficient condition on $\mu$ so that $(K,\mu)$ satisfy a Bernstein-Markov property? In the univariate case, much is known but many questions remain. One would expect that if $(K,\mu)$ satisfy a Bernstein-Markov property, then $\mu$ should be sufficiently dense on $K$, or at least on its support, supp$(\mu)$. We call a Borel subset $E$ of $K$ a {\it carrier} of $\mu$ if $\mu(E)=\mu(K)$. One can define extremal functions $V_E$ of Borel sets $E$ generalizing (\ref{veekpoly}) in the following way. Note that for a polynomial $p$, the function 
$u(z)=\frac{1}{deg (p)}\log |p(z)|$ is psh in $\C^n$ and of logarithmic growth at infinity. We let $L(\C^n)$ denote the set of all $u$ psh in $\C^n$ with 
$$u(z)\leq \log |z|+0(1), \ |z|\to \infty.$$
Then for $E\subset \C^n$, define
\begin{equation}\label{vkno}V_E(z):=\sup \{u(z): \ u\in L(\C^n), \ u\leq 0 \ \hbox{on} \ E\}.\end{equation}
If $K$ is compact, this agrees with (\ref{veekpoly}) (cf., \cite{K}); indeed, if $K$ is compact and $Q$ is an admissible weight on $K$ then $V_{K,Q}$ in (\ref{veekqpoly}) agrees with
 \begin{equation}\label{vkq} V_{K,Q}(z)=\sup \{u(z): \ u\in L(\C^n), \ u\leq Q \ \hbox{on} \ K\}\end{equation}
(cf., Appendix B in \cite{SaTo}).

 Following Ullman's work on real intervals (cf., p. 102 of \cite{ST} or \cite{UW}) we call $\mu$ a {\it determining measure} for $K$ if for all carriers $E$ of $\mu$, 
$V_E^*=V_K^*$. To avoid trivialities, one should assume $K$ is not pluripolar. If $K$ is a {\it regular} compact set, which means that $V_K=V_K^*$; i.e., $V_K$ is continuous, it is known that {\it if $\mu$ is a determining measure for $K$ then $(K,\mu)$ satisfy a Bernstein-Markov property} (cf., \cite{Biumj}). The converse is far from true: {\it any} compact set admits a Bernstein-Markov measure with a countable (and hence pluripolar) carrier. The following is from \cite{PELD}. 

\begin{proposition} \label{universal} Let $K\subset \C^n$ be an arbitrary compact set. Then there exists a finite measure $\nu$ with support in $K$ which is carried by a countable set such that $(K,\nu)$ satisfies a Bernstein-Markov property.
 \end{proposition}
 
 \begin{proof} If $K$ is a finite set, any measure $\nu$ which puts positive mass at each point of $K$ will work. If $K$ has infinitely many points, letting $m_k :=$dim$\mathcal P_k(K)$, where $\mathcal P_k(K)$ denotes the polynomials of degree at most $k$ restricted to $K$, we have $m_k\to \infty$ and $m_k\leq N_k={n+k\choose k}=0(n^k)$. For each $k$, let 
 $$\mu_k :=\frac{1}{m_k} \sum_{j=1}^{m_k} \delta (z_j^{(k)})$$
 where $\{z_j^{(k)}\}_{j=1,...,m_k}$ is a set of Fekete points of order $k$ for $K$ relative to the vector space $\mathcal P_k(K)$; i.e., if $\{e_1,...,e_{m_k}\}$ is any basis for $\mathcal P_k(K)$, 
\begin{equation}\label{fekhere}\bigl |VDM(z_1^{(k)},...,z_{m_k}^{(k)})\bigr|:=\bigl |\det [e_i(z_j^{(k)})]_{i,j=1,...,m_k}\bigr|=\max_{q_1,...,q_{m_k}\in K}\bigl |\det[e_i(q_j)]_{i,j=1,...,m_k}\bigr|.\end{equation}
 Define
 $$\nu:=\sum_{k=1}^{\infty} \frac{1}{k^2}\mu_k.$$
 If $p\in \mathcal P_k(K)$, we have
 $$p(z)=\sum_{j=1}^{m_k} p(z_j^{(k)}) l_j^{(k)}(z)$$
 where 
 $$l_j^{(k)}(z)=\frac{VDM(z_1^{(k)},...,z_{j-1}^{(k)},z,z_{j+1}^{(k)},...,z_{m_k}^{(k)})}{VDM(z_1^{(k)},...,z_{m_k}^{(k)})}\in \mathcal P_k(K)$$ with $l_j^{(k)}(z_k^{(k)})=\delta_{jk}$. We have  
 $||l_j^{(k)}||_K=1$ from (\ref{fekhere}) and hence
 $$||p||_K \leq \sum_{j=1}^{m_k} |p(z_j^{(k)})|.$$
 On the other hand, 
 $$||p||_{L^1(\nu)}\geq  \frac{1}{k^2}\int_K |p|d\mu_k$$
 $$= \frac{1}{k^2m_k}\sum_{j=1}^{m_k} |p(z_j^{(k)})|.$$
 Thus we have
 $$||p||_K \leq k^2m_k ||p||_{L^1(\nu)}\le \sqrt{\nu(K)}k^2m_k ||p||_{L^2(\nu)}.$$
\end{proof}

\begin{remark} Fekete points for a general compact set $K$ are difficult to find; thus the previous result is nonconstructive. For a more constructive approach, let $\{{\bf A_k}\}_{k=1,...}:=\{a_1^{(k)},...,a_{M_k}^{(k)}\}_{k=1,...}$ with $a_j^{(k)}\in K$ and $M_k\geq m_k$ be a {\it weakly admissible mesh} (WAM) for $K$: this means that \begin{equation}\label{wam}||p||_K \leq C_k ||p||_{{\bf A_k}} \ \hbox{for all} \ p\in \mathcal P_k(K)
\end{equation}
where both sequences $\{M_k\}$ and $\{C_k\}$ grow polynomially in $k$. Such sets are often very easy to construct (cf., \cite{calvilev}). Replacing the Fekete points $\{z_j^{(k)}\}_{j=1,...,m_k}$ of order $k$ for $K$ in the previous proof by Fekete points $\{f_j^{(k)}\}_{j=1,...,m_k}$ of order $k$ for the finite set ${\bf A_k}$, and defining the measures $\mu_k :=\frac{1}{m_k} \sum_{j=1}^{m_k} \delta (f_j^{(k)})$, using (\ref{wam}) it is straightforward to see that $\nu:=\sum_{k=1}^{\infty} \frac{1}{k^2}\mu_k$ satisfies the estimate 
$$||p||_K \leq C_kk^2m_k ||p||_{L^1(\nu)} \ \hbox{for all} \ p\in \mathcal P_k(K),$$
verifying that $(K,\nu)$ satisfies a Bernstein-Markov property. Moreover, there are known algorithms for computing {\it approximate Fekete points} (cf., \cite{SV}). 

\end{remark}

We illustrate how a simple {\it mass density condition} on a measure is sufficient for a Bernstein-Markov property. The argument we give is modeled on that given in \cite{BLmass} and goes back to Ullman in the univariate case. Below, $B(z,r):=\{w \in \C^n:|w-z|<r\}$.

\begin{proposition} \label{mdprop} Let $\mu$ be a measure such that supp$(\mu)=K$ is a regular compact set and $\mu$ satisfies the following mass density condition: there exists $r_0>0$ and $T>0$ with $\mu(B(z,r))\geq r^T$ for all $z\in K$ and $r<r_0$.  Then $(K,\mu)$ satisfy a Bernstein-Markov property.
\end{proposition}

\begin{proof} Given $\epsilon >0$, choose $\delta >0$ so that $V_K(z) \leq \epsilon$ if dist$(z,K)\leq \delta$. From the definition of $V_K$, for such $z$ we have for $p_k\in \mathcal P_k$, 
 \begin{equation}\label{sup} |p_k(z)|\leq ||p_k||_K \cdot e^{k\epsilon}.\end{equation}
We consider $k$ sufficiently large so that $r_k:=(\delta/3)e^{-k\epsilon}<r_0$. Given $p_k\in \mathcal P_k$ we fix $w\in K$ with $|p_k(w)|=||p_k||_K$. We show that 
\begin{equation}\label{supnorm} 
|p_k(z)|\geq \frac{1}{2}|p_k(w)| \ \hbox{for} \ z\in K \ \hbox{with} \ |z-w|<
r_{k}.
\end{equation}
Suppose we have proved (\ref{supnorm}). To verify the Bernstein-Markov property, we first observe that 
$$\mu (K\cap B(w,r_k))\geq r_k^T.
$$
Then
$$\int_K |p_k(z)|d\mu(z)\geq \int_{K\cap B(w,r_k)}|p_k(z)|d\mu(z)$$
$$\geq \min_{z\in K\cap B(w,r_k)} |p_k(z)|\cdot \mu (K\cap B(w,r_k))\geq \frac{1}{2}|p_k(w)| \cdot (\delta/3)^{T}e^{-kT\epsilon},$$
verifying (\ref{used}) with $L^1$ norm.

To prove (\ref{supnorm}), for fixed $z\in K$ with $|z-w|<r_{k}
$ consider $$U(t):=p_k(w+at) \ \hbox{where} \ a=\frac{z-w}{|z-w|}.$$
Then $U(0)=p_k(w)$ and $U(|z-w|) = p_k(z)$. Furthermore, if $|t|<\delta$, then 
dist$(w+at,K)\leq \delta$ so that $V_K(w+at) \leq \epsilon$ and from (\ref{sup}),
\begin{equation}\label{last}|U(t)|=|p_k(w+at)|\leq |p_k(w)|\cdot e^{k\epsilon}.\end{equation}
We apply the Cauchy estimates to $U'(t)$ in $|t|<r_{k}$ to obtain
$$||U'||_{ |t|<r_{k}}\leq \frac{3}{2\delta} ||U||_{ |t|<\delta}.$$
Using
$$U(|z-w|)-U(0)=\int_0^{|z-w|}U'(t)dt$$
we find that for $|z-w|<r_{k}
$ with $z\in K$ (using (\ref{last})) 
$$|p_k(z)-p_k(w)|\leq |z-w|\cdot ||U'||_{ |t|<r_{k}}\leq 
\frac{\delta}{3}\cdot e^{-k\epsilon}\cdot \frac{3}{2\delta} ||U||_{ |t|<\delta}$$
$$\leq \frac{1}{2} e^{-k\epsilon}\cdot |p_k(w)|\cdot e^{k\epsilon}=\frac{1}{2} |p_k(w)|$$
which proves (\ref{supnorm}). 
\end{proof}

An important observation is that to obtain equation (\ref{sup}) at the very beginning, in conjunction with regularity of $K$ we have used the {\it Bernstein-Walsh estimate}: for a compact set $K$, a polynomial $p$ satisfies
\begin{equation}\label{bwin} |p(z)|\leq ||p||_K\cdot e^{deg (p) \cdot V_K(z)}, \ z\in \C^n. \end{equation}
This is valid in any number of dimensions and follows from (\ref{veekpoly}).

In one variable the property that a measure $\mu$ be determining; i.e., for all carriers $E$ of $\mu$, $V_E^*=V_K^*$, can be re-interpreted as the capacity $C(E)$ of each carrier is equal to that of $K$; i.e., $\mu(E) =\mu(K)$ implies $C(E)=C(K)$. Here, the inner and outer capacities associated to the capacitary set function defined for compact sets in (\ref{cap}) coincide for a Borel set $E$; it is this notion of capacity for Borel sets we use (see Appendix A.I in \cite{ST}). In $\C^n$ for $n>1$, this re-interpretation remains valid provided we use the notion of {\it relative capacity}: fixing a ball $B$ containing $K$, for $E$ a Borel subset of $B$, define
\begin{equation}\label{relcap}C(E):=\sup \{\int_E (dd^cu)^n: u \ \hbox{psh in} \ B, \ 0\leq u \leq 1\}\end{equation}
(we use the same notation although the notions are different for $n=1$ and $n>1$). 
There are stronger results than Proposition \ref{mdprop} which show
that one only needs a type of denseness ``in the mean'' with respect to capacity. The following result is Theorem 4.2.3 of \cite{ST} in the univariate case while the exact analogue in several complex variables with the relative  capacity is Theorem 2.2 in \cite{BLmass}. 

\begin{theorem} \label{suffmd} Let $\mu$ be a measure such that supp$(\mu)=K$ is a regular compact set and suppose there exists $T>0$ such that
\begin{equation}\label{md??} \lim_{r\to 0^+} C(\{z:\mu(B(z,r)) \geq r^T\})=C(K).\end{equation}
Then $(K,\mu)$ satisfy a Bernstein-Markov property.
\end{theorem}

\begin{proof} We just indicate modifications needed from the proof of Proposition \ref{mdprop}; for details, we refer the reader to pp. 112-115 of \cite{ST} for the univariate case and pp. 4764-4765 of \cite{BLmass} for the multivariate version. By taking $\delta >0$ smaller, if necessary, {\it for any compact subset $E$ of $K$ with $C(E)>C(K) -\delta$ we have $V_E^*\leq V_K +\epsilon$.} This statement is one of the main points of \cite{BLmass} in $\C^n$  for $n>1$ with the relative capacity (\ref{relcap}). The Bernstein-Walsh estimate (\ref{bwin}) for $E$ gives
\begin{equation}\label{bwin2} ||p||_K\leq ||p||_Ee^{\epsilon \cdot deg(p)}. \end{equation}
We choose such an $E$ so that, in addition, for $k$ sufficiently large ($r_k$ in the proof of Proposition \ref{mdprop} sufficiently small),
$$\mu(B(z,r_k)) \geq r_k^T \ \hbox{if} \ z\in E.$$
This follows from the mass density hypothesis (\ref{md??}). For $p_k\in \mathcal P_k$, we choose $w\in E$ so that $|p_k(w)|=||p_k||_E$ and we obtain the analogue of (\ref{supnorm}):
$$
|p_k(z)|\geq \frac{1}{2}|p_k(w)| \ \hbox{for} \ z\in E \ \hbox{with} \ |z-w|<
r_{k}.
$$
Using (\ref{bwin2}) for $p=p_k$, the rest of the proof proceeds as in Proposition \ref{mdprop}.
\end{proof}

 \begin{remark} There are no known examples of measures $\mu$ with (regular) support $K$ such that $(K,\mu)$ satisfy a Bernstein-Markov property but for which $\mu$ does {\bf not} satisfy (\ref{md??}) for any $T>0$. Letting $t:=e^{-T}$, the sufficient condition (\ref{md??}) can be rewritten as
$$\exists t<1,\quad C(\{z: \mu(B(z,r))\geq t^{\log(1/r)}\})\to C(K) \quad\text{as }r\to 0.$$
If the regular set $K$ is the interval $[0,1]$, Theorem 4.2.8 of \cite{ST} gives a {\it necessary} condition for $(K,\mu)$ to satisfy a Bernstein-Markov property:
 $$\forall t<1,\quad C(\{z: \mu(B(z,r))\geq t^{1/r}\})\to  C(K) \quad\text{as }r\to 0.$$
Thus there is a gap between these conditions; i.e., the logarithmic vs. the linear rates in the exponent of $t$.
 \end{remark}
\medskip

What about sufficient conditions for $(K,\mu,Q)$ to satisfy a {\it weighted} Bernstein-Markov property? Suppose $K\subset \C^n$ is compact and nonpluripolar and $\mu$ is a measure on $K$ with $(K,\mu)$ satisfying a Bernstein-Markov property. For appropriate pairs $(K,\mu)$, it may be the case that for {\it any} continuous weight function $w=e^{-Q}$ on $K$, the triple $(K,\mu,Q)$ satisfies a weighted Bernstein-Markov property. If this is the case, we call $\mu$ a {\it strong Bernstein-Markov measure} for $K$. We have the following. 

\begin{proposition} \label{mdprop2} Let $\mu$ be a measure such that supp$(\mu)=K$ is a compact set with the property that for all $z\in K$, $K\cap \bar B(z,r)$ is a regular compact set for $r$ sufficiently small. Suppose $\mu$ satisfies the mass density condition of Proposition \ref{mdprop}: there exists $r_0>0$ and $T>0$ with $\mu(B(z,r))\geq r^T$ for all $z\in K$ and $r<r_0$. Then $\mu$ is a strong Bernstein-Markov measure for $K$.
\end{proposition}

\begin{remark} \label{mdremark} We remark that if $n=1$, any regular compact set $K\subset \C$ with $\C\setminus K$ connected satisfies the condition that for all $z\in K$, $K\cap \bar B(z,r)$ is a regular compact set for $r$ sufficiently small. The proof of the proposition is a modification of the proofs of Proposition \ref{mdprop} and Theorem \ref{suffmd}. We refer the reader to Theorem 4.13 of \cite{BLTW} where a generalization of the above result in the univariate case is proved under somewhat weaker hypotheses on $K$.
\end{remark}

Recall from the introduction that if $K'=\{z: |z| \leq 1 \}$ and $d\nu = \frac{1}{2\pi} d\theta$ on the boundary, the function $Q(z) = |z|^2$ gives an example where $(K',\nu,Q)$ does not satisfy a weighted Bernstein-Markov property. Here, the mass density condition is not satisfied (it is satisfied for $K=\{z:|z|=1\}$).

\begin{remark} \label{ldp2} The strong Bernstein-Markov property is useful in proving {\it large deviation principles} associated to probability distributions arising in (pluri-)potential theoretic frameworks. In the setting of Remark \ref{ldp}, if $\nu$ is a  strong Bernstein-Markov measure on $K$, a large deviation principle says, roughly, that for each $Q$ continuous on $K$, given an array $X=\{z_j^{(k)}\}\in \chi$, the probability that the sequence of measures $\{ \nu_k\}$ in (\ref{nuk}) satisfies $\nu_k \to \mu\neq \mu_{K,Q}$ goes to $0$ like $e^{-k^2[I^Q(\mu)- I^Q(\mu_{K,Q})]}$. Here, $\mathcal I(\mu):=I^Q(\mu)- I^Q(\mu_{K,Q})\geq 0$ is the {\it rate function} and $k^2$ is the {\it speed}. Thus one gets a quantitative description of how much the sequence of measures $\{ \nu_k\}$ associated to a random array in $K$ as in Remark \ref{ldp} can deviate from the equilibrium measure $\mu_{K,Q}$. There are analogous results in the multidimensional case. We refer the reader to \cite{PELD} and \cite{LPLD} for more information and further references.
\end{remark}

A weighted situation that arises naturally in the univariate case is when  
$$Q(z)=-U_{\nu}(z)=\int_P \log |z-\zeta| d\nu(\zeta)$$
is minus the logarithmic potential of a measure $\nu$ supported in a set $P$ disjoint from $K$. Note since $Q$ is harmonic outside $P$ it is, in particular, continuous on $K$. This will appear in the next section (cf., Proposition \ref{wtvsrat}).

\section{Rational Bernstein-Markov in $\C$} 

In the next two sections, we work exclusively in $\C$. Let $P$ be a fixed compact set in $\C$. For $k=1,2,...$ let $\mathcal R_k(P)$ denote the rational functions $r_k=p_k/q_k$ where $p_k,q_k\in \mathcal P_k$ and all zeros of $q_k$ lie in $P$. Following \cite{P}, for $K\subset \C$ compact with $K\cap P=\emptyset$ and a measure $\mu$ on $K$ we will say that $\{K,\mu,P\}$ has the {\it rational Bernstein-Markov property} if for any sequence $\{r_k\}$ with $r_k\in \mathcal R_k(P)$, 
\begin{equation}\label{rmk} \limsup_{k\to \infty} \bigg(\frac{||r_k||_K}{||r_k||_{L^2(\mu)}}\bigg)^{1/k} \leq 1.\end{equation}

We have the following \cite{P}.

\begin{proposition}\label{wtvsrat} Let $K\subset \C$ be a nonpolar compact set and $\mu$ a measure on $K$. Let $P$ be a compact set with $K\cap P=\emptyset$. The following are equivalent:
\begin{enumerate}
\item $\{K,\mu,P\}$ has the rational Bernstein-Markov property; 
\item For all measures $\sigma$ on $P$ with mass at most one, and corresponding potentials $U_{\sigma}$, $(K,-U_{\sigma},\mu)$ has the weighted (polynomial) Bernstein-Markov property.

\end{enumerate}
\end{proposition}

The connection between these two properties is that if one takes $r_k=p_k/q_k$ with $p_k,q_k\in \mathcal P_k$ and $q_k(z)=\prod_{j=1}^{m_k}(z-z_j)$ where $z_j\in P$ and $m_k\leq k$, then 
$$|r_k(z)|=|p_k(z)|\cdot \exp {[k U_{\mu_k}(z)]}$$
for $\mu_k:=\frac{1}{k}\sum_{j=1}^{m_k} \delta(z_j)$. Thus we have written $r_k$ as a weighted polynomial of degree $k$ with $Q=-U_{\mu_k}$ which is continuous on $K$. 

Given a compact set $K\in \C$, 
$$\hat K:=\{z\in \C: |p(z)|\leq ||p||_K \ \hbox{for all polynomials} \ p\}$$ 
is the {\it polynomial hull} of $K$; it is simply the complement of the unbounded component $\Omega_K$ of $\C\setminus K$. If $K=\{z:|z|=1\}$, then $\hat K = K'=\{z: |z| \leq 1 \}$. We let $S_K$ denote the Shilov boundary of $K$ with respect to the uniform algebra $P(K)$ of uniform limits of polynomials restricted to $K$; thus $S_K$ is the smallest closed subset $S$ of $K$ such that for every polynomial $p$, 
$||p||_S = ||p||_K$. For $K'=\{z: |z| \leq 1 \}$, $S_{K'}=\{z:|z|=1\}$. Using Proposition \ref{wtvsrat}, one can show the following \cite{P}.

\begin{theorem} \label{rbmp} Let $K\subset \C$ be a nonpolar compact set and $\mu$ a measure with support $K$ such that $(K,\mu)$ has the (polynomial) Bernstein-Markov property. Let $P$ be a compact set with $K\cap P=\emptyset$. If either $S_K=K$ or $\hat K \cap P=\emptyset$, then $\{K,\mu,P\}$ has the rational Bernstein-Markov property.

\end{theorem}

For a compact set $K$ satisfying the assumptions of Theorem \ref{rbmp}, it follows from Theorem \ref{suffmd} that if $K$ is regular, the same sufficient mass density condition (\ref{md??}) on  a measure $\mu$ with support $K$ implies that $\{K,\mu,P\}$ has the rational Bernstein-Markov property. However, even if $\hat K \cap P\not = \emptyset$ and $S_K\not = K$, one still has the following result \cite{P}.
\begin{theorem} Let $K\subset \C$ be regular and let $P$ be compact with $\hat P \cap K=\emptyset$. There exist points $w_1,...,w_m\in \C\setminus (K\cup \hat P)$ and positive numbers $R_1<R_2$ such that 
$$K\subset \subset \{|f|<R_1\} \ \hbox{and} \ P \subset \subset \{R_1< |f|<R_2\}$$
where $f(z):=\prod_{j=1}^m (z-w_j)^{-1}$. Given a measure $\mu$ with support $K$, if there exists $T>0$ such that
$$ \lim_{r\to 0^+} C(\{z\in f(K):f_*\mu(B(z,r)) \geq r^T\})=C(f(K)),$$
then $\{K,\mu,P\}$ has the rational Bernstein-Markov property. 
\end{theorem}
\noindent Here $f_*\mu$ is the push-forward of $\mu$ under $f$ on $f(K)$.

In the weighted situation, for $K\subset \C$ compact with $Q$ admissible, $\nu$ a measure on $K$, and $P\cap K=\emptyset$, we say that $\{K,\nu,Q,P\}$ satisfy a {\it weighted rational Bernstein-Markov property} if for all $r_k\in \mathcal R_k(P)$, 
$$||e^{-kQ}r_k||_K \leq M_k ||e^{-kQ}r_k||_{L^2(\nu)} \ \hbox{with} \ \limsup_{k\to \infty} M_k^{1/k} =1.$$ 
Again, one can replace $L^2(\nu)$ by $L^p(\nu)$ for $1<p<\infty$. We give an application of (weighted) rational Bernstein-Markov properties. In \cite{VELD}, the following vector potential problem is discussed: given a $d$-tuple of nonpolar compact sets $K=(K_{1},\ldots,K_{d})$ in $\C$ and given a symmetric positive semidefinite matrix
$$C:=(c_{i,j})_{i,j=1}^{d},$$
the associated unweighted energy of a $d$-tuple of measures $\mu=(\mu_{1},\ldots,\mu_{d})$ is defined as
$$E(\mu):=\sum_{i,j=1}^{d}c_{i,j}I(\mu_{i},\mu_{j}),$$
where $I(\mu_{i},\mu_{j})$ is the mutual energy:
$$I(\mu_{i},\mu_{j})=\int\int \log\frac{1}{|z-t|}d\mu_{i}(z)d\mu_{j}(t).$$
Given a $d$-tuple of admissible weights $Q=(Q_{1},\ldots,Q_{d})$ with $Q_{i}$ defined on $K_{i}$, $i=1,\ldots,d$, the weighted energy of $\mu=(\mu_{1},\ldots,\mu_{d})$ is defined as
$$E_Q(\mu):=E(\mu)+2\sum_{i=1}^{d}\int Q_{i}d\mu_{i}.$$
We minimize $E_Q(\mu)$ over an appropriate class of $d$-tuples of measures $\mu$: fixing $r_1,\ldots,r_d>0$ set 
\begin{equation*}
\mathcal M_r(K):=\{\mu=(\mu_{1},\ldots,\mu_{d}), \quad\mu_i\in \mathcal M_{r_i}(K_i), \ i=1,\ldots,d
\}
\end{equation*} where $\mathcal M_{r_i}(K_i)$ denotes the positive measures on $K_i\subset \C$ of mass $r_i$. In discretizing $E_Q(\mu)$, because some of the $c_{i,j}$ can be negative, one is led to study vector versions of (weighted) rational Bernstein-Markov properties in order to get generalizations of Propositions \ref{znunwtd} and \ref{weightedtd} in this setting. Indeed, to discretize the vector energies $E$ and $E_{Q}$, for each $k=1,2,...$,  take a sequence of ordered tuples  $m_{k}=(m_{1,k},\ldots,m_{d,k})$ of positive integers with 
\begin{equation}\label{correct2}
m_{i,k}\uparrow \infty,\quad i=1,\ldots,d, \quad \hbox{and} \quad \lim_{k\to \infty}  \frac{m_{i,k}}{m_{{j,k}}}=\frac{r_i}{r_j},\quad i,j=1,\ldots,d
\end{equation} 
so that for $k$ large
\begin{equation}\label{correct3}\frac{r_i^2}{m_{i,k}^2}\asymp \frac{r_j^2}{m_{j,k}^2} \asymp \frac{r_ir_j}{m_{i,k}m_{j,k}}.
\end{equation}
For a set of (distinct) points of the form
\begin{equation}\label{Zk}
{\bf Z_{k}}=\cup_{i=1}^{d}\{z_{i,1},\ldots,z_{i,m_{i,k}}\in K_{i}\},
\end{equation}
 let
\begin{equation}\label{vdm}
|{\bf VDM_k}({\bf Z_{k}})|:=\prod_{i=1}^{d}\prod_{l<p}^{m_{i,k}}|z_{i,l}-z_{i,p}|^{c_{i,i}}
\cdot \prod_{i<j}^{d}\prod_{l=1}^{m_{i,k}}\prod_{p=1}^{m_{j,k}}|z_{i,l}-z_{j,p}|^{c_{i,j}}.
\end{equation}
We define a $k-$th order vector diameter with respect to $(m_{1,k},\ldots,m_{d,k})$ and a sequence satisfying (\ref{correct2}) as
\begin{equation}\label{diam}
\delta^{(k)}(K):=\max_{{\bf Z_{k}}}\Big[|{\bf VDM_k}({\bf Z_{k}})|\Big]^{2|r|^{2}/|m_{k}|(|m_{k}|-1)}.
\end{equation}
Here $|r|=r_{1}+\cdots+r_{d},\quad |m_{k}|=m_{1,k}+\cdots+m_{d,k}$. For the weighted case define  
$$|{\bf VDM_k}^Q({\bf Z_{k}})|:=|{\bf VDM_k}({\bf Z_{k}})|\cdot 
\prod_{i=1}^{d}\prod_{l=1}^{m_{i,k}} e^{-\frac{m_{i,k}}{r_i}Q_i(z_{i,l})}$$ 
and the $k-$th order weighted vector diameter
\begin{equation}\label{diam-w}
\delta^{(k)}_Q(K):=\max_{{\bf Z_{k}}}\Big[|{\bf VDM_k}^Q({\bf Z_{k}})|\Big]^{2|r|^{2}/|m_{k}|(|m_{k}|-1)}.
\end{equation}
The factor $|m_{m}|(|m_{k}|-1)/2$ in the exponent of (\ref{diam}) and (\ref{diam-w}) corresponds to the number of factors in the product (\ref{vdm}). This has applications to {\it multiple orthogonal polynomial} ensembles.

For the appropriate Bernstein-Markov properties, fix $0<p<\infty$ and let $\nu=(\nu_1,\ldots,\nu_d)$ be a tuple of measures with $\nu_i$ supported in $K_i$ for $i=1,\ldots,d$. Let 
$${\mathcal R}^i_k=\{r_k=p_k/q_k: p_k, q_k\in \mathcal P_k; \ \hbox{all zeros of} \ q_k \ \hbox{lie in} \ \cup_{j\neq i}K_j\}.$$
Writing $K=(K_1,...,K_d)$, we say $(K,\nu)$ satisfies an {\it $L^p-$rational Bernstein-Markov property} if for $i=1,\ldots,d$,
$$||r_k||_{K_i}\leq M^{(p)}_{k,i} ||r_k||_{L^p(\nu_i)}, \ r_k\in \mathcal R^i_k$$
where $(M^{(p)}_{k,i})^{1/k}\to 1$ as $k\to\infty$; i.e., each $\{K_i,\nu_i,\cup_{j\neq i}K_j\}$ satisfies an $L^p-$rational Bernstein-Markov property. More generally, let $Q=(Q_1,\ldots,Q_d)$ be a $d-$tuple of admissible weights for $K=(K_1,...,K_d)$. We say $(K,\nu,Q)$ satisfies an {\it $L^p-$weighted rational Bernstein-Markov property} if  for $i=1,\ldots,d$,
$$||r_ke^{-kQ_i}||_{K_i}\leq M^{(p)}_{k,i} ||r_ke^{-kQ_i}||_{L^p(\nu_i)}, \ r_k\in \mathcal R^i_k$$
where $(M^{(p)}_{k,i})^{1/k}\to 1$ as $k\to\infty$; i.e., each $\{K_i,\nu_i,Q_i,\cup_{j\neq i}K_j\}$ satisfies an $L^p-$ weighted rational Bernstein-Markov property. These notions are utilized in \cite{VELD} to first prove a version of Proposition \ref{znunwtd} with 
$$Z^Q_k:= \int_{K^{k}}|{\bf VDM_k}^Q({\bf Z_k})|^2d\nu({\bf Z_k})$$
and eventually a large deviation principle.

\section{Extensions in the univariate setting}

The techniques of using a Bernstein-Walsh estimate plus a mass density condition to prove a Bernstein-Markov property as in Proposition \ref{mdprop} can be applied in other situations. We discuss a situation occurring in the univariate setting. The following is in \cite{BLTW}.

We let $K\subset \C$ be a nonpolar compact set and we let $f:K \to \C$ be continuous. Consider the weighted potential theory problem of minimizing 
\begin{equation}\label{def-EQ}
E^Q_f(\mu)=E^Q(\mu):= \int_{K}\int_{K}\log\frac{1}{|x-y||f(x)-f(y)|w(x)w(y)}d\mu(x)d\mu(y)
\end{equation}
over $\mu \in \mathcal M(K)$ where $w=e^{-Q}\geq 0$ is usc on $K$ with
$\{z\in{K}: Q(z)<\infty\}$ nonpolar (compare with the discussion before (\ref{wten})). Indeed, we observe that the energy $E^{Q}(\mu)$ in (\ref{def-EQ}) can be rewritten as
$$E^{Q}(\mu)=I^Q(\mu) + I(f_*\mu)$$
if all terms exist where 
$$I(f_*\mu) =\int_K \int_K \log \frac{1}{|f(x)-f(y)|}d\mu(x) d\mu(y)=\int_{f(K)}\int_{f(K)}\log \frac{1}{|a-b|}df_*\mu(a) df_*\mu(b)$$
is the logarithmic energy of the push-forward $f_*\mu$ of $\mu$ on $f(K)$. Provided there exists one $\mu \in \mathcal M(K)$ with $E^Q(\mu) < \infty$, it is shown in \cite{BLTW} that there then exists a unique minimizing measure $\mu_{K,Q}^f\in \mathcal M(K)$ of (\ref{def-EQ}). Let 
  \begin{equation}\label{fvdm} |{\it fVDM_k^Q}(z_0,...,z_k)|:= \end{equation}
$$|VDM_k(z_0,...,z_k)|\exp \Big(-k[Q(z_0)+\cdots + Q(z_k)]\Big)|VDM_k(f(z_0),...,f(z_k))| $$
where $VDM_k(z_0,...,z_k)=\det [z_j^i]_{i,j=0,...,k}$ as in (\ref{deltak2}). Setting 
  $$\bigl(\delta_k^Q(f)\bigr)(K):=\max_{z_0,...,z_k\in K}  |{\it fVDM_k^Q}(z_0,...,z_k)|^{2/k(k+1)},$$
we have 
$$\bigl(\delta^Q(f)\bigr)(K):=\lim_{k\to \infty} \bigl(\delta_k^Q(f)\bigr)(K)=\exp{(-E^Q(\mu_{K,Q}^f))}.$$

We now assume, for simplicity, that $f:K \to \C$ extends to a nonconstant entire function (weaker hypotheses which suffice for what follows can be found in \cite{BLTW}). A particular example that occurs in the literature is $f(z)=e^z$. Then $$\{z\in{K}: ~f'(z)\neq 0\text{ and }Q(z)<\infty\}$$ is nonpolar; this implies there exists $\mu \in \mathcal M(K)$ with $E^Q(\mu) < \infty$. In this setting, we consider 
$$\mathcal{F}_k:=\{h_k: h_k(z)=p_k(z)q_k(f(z)) \ \hbox{where} \ p_k, q_k\in \mathcal P_k\}.$$
For $K\subset \C$ compact we define an extremal function \begin{equation}\label{wk}
 W_K(z):=\sup\{\frac{1}{k}\log |h_k(z)|: \ h_k\in\mathcal{F}_k \quad\text{and}\quad ||h_k||_K\leq 1\}, \ z\in \C
 \end{equation}
 and with $Q$ as above we define a weighted extremal function \begin{equation}\label{wkq}
 W_{K,Q}(z):=\sup\{\frac{1}{k}\log |h_k(z)|: \ h_k\in\mathcal{F}_k \quad\text{and}\quad ||h_ke^{-kQ}||_K\leq 1\}, \ z\in \C.
 \end{equation}
 Using some logarithmic potential theory (cf., \cite{R}), we can show that $W_{K,Q}^*$ is subharmonic and, if $K$ is regular, then $W_K^*=0$ on $K$. We obtain a Bernstein-Walsh type estimate for functions $h_k\in \mathcal{F}_k$:
\begin{equation}\label{BW}
|h_k(z)|\leq ||h_k||_Ke^{kW_K^*(z)}, \ z\in \C
\end{equation}
as well as a weighted Bernstein-Walsh type estimate:
\begin{equation}\label{BWQ}
|h_k(z)|\leq ||h_ke^{-kQ}||_Ke^{kW_{K,Q}^*(z)}, \ z\in \C.
\end{equation}

 Given a measure $\mu$ on $K$, we say that $(K,\mu,Q)$ satisfies a weighted Bernstein-Markov property for $\mathcal{F}_k$ if for all $h_k\in\mathcal{F}_k$ we have
$$||h_ke^{-kQ}||_K\leq M_k\int_K |h_k(z)|e^{-kQ(z)}d\mu (z)  \ \hbox{with} \ \limsup_{k\to \infty} M_k^{1/k} =1.
$$
If $\mu$ satisfies a weighted Bernstein-Markov inequality for $\mathcal{F}_k$ for all continuous $Q$ on $K$, we say $\mu$ satisfies a strong Bernstein-Markov inequality for $\mathcal{F}_k$. Using our Bernstein-Walsh estimates (\ref{BW}) and (\ref{BWQ}), one obtains the following result, a special case of Theorem 4.13 of \cite{BLTW} (alluded to in our Remark \ref{mdremark}):

\begin{theorem}\label{BM} Suppose that $\mu$ is a Borel measure $\mu$ on the regular compact set $K$ with the property that for all $z\in K$, $K\cap \bar B(z,r)$ is a regular compact set for $r$ sufficiently small. If $\mu$ satisfies the mass density condition: there exists $r_0>0$ and $T>0$ with $\mu(B(z,r))\geq r^T$ for all $z\in K$ and $r<r_0$, then $\mu$ is a strong Bernstein-Markov measure for $\mathcal{F}_k$ on $K$.
\end{theorem}

\noindent One can utilize Theorem \ref{BM} to obtain probabilistic results involving equidistribution of random arrays of points as well as a large deviation principle \cite{BLTW}.

\section{Riesz energies}

In this section we work in Euclidean space $\R^d$ for any $d\geq 2$ and we consider the $\alpha-$Riesz kernel
$$W(x-y)=W_{\alpha}(x-y):= \frac{1}{|x-y|^{\alpha}}, \ x,y\in \R^d$$
for a fixed $\alpha \in (0,d]$. The Riesz energy associated to a measure $\mu$ on a set $K\subset \R^d$ is
$$\int_K \int_K \frac{1}{|x-y|^{\alpha}} d\mu(x) d\mu(y)$$
and for $Q:K\to \R$, the weighted Riesz energy is
$$\int_K \int_K \frac{1}{|x-y|^{\alpha}} d\mu(x) d\mu(y)+ 2\int_K Q(x)d\mu(x).$$
Discretizing this weighted energy, for $n\geq 2$, let 
$$VDM_n^Q(x_1,...,x_n):= \exp {\bigl[-\sum_{1\leq i \neq j\leq n}W(x_i-x_j)\bigr]} \cdot \exp{\bigl[-2n\sum_{j=1}^nQ(x_j)\bigr]}$$
$$=\exp {\bigl[-\sum_{1\leq i \neq j\leq n}\frac{1}{|x_i-x_j|^{\alpha}}\bigr]} \cdot \exp{\bigl[-2n\sum_{j=1}^nQ(x_j)\bigr]}.$$

If we fix $n-1$ points $x_2,...,x_n$ and consider
$$y\to  VDM_n^Q(y,x_2,...,x_n), \ y=(y_1,...,y_d)\in \R^d,$$
then this function is of the form 
$$g(x_2,...,x_n)\cdot f_n^Q(y):=g(x_2,...,x_n)\cdot \exp  {\bigl[-\sum_{j=2}^n\frac{1}{|y-x_j|^{\alpha}}-2nQ(y)\bigr]}$$
where 
$$g(x_2,...,x_n)=\exp{\bigl(-2\sum_{j=2}^n Q(x_j)}\bigr)\cdot \exp  {\bigg(-\sum_{2\leq j \neq k\leq n }\frac{1}{|x_k-x_j|^{\alpha}}\bigg)}.$$ 
We will prove a {\it Bernstein-type estimate} in Proposition \ref{derivest} for derivatives of such functions $f_n^Q(y)$. This will be used as a substitute for a Bernstein-Walsh estimate (\ref{bwin}) in proving Proposition \ref{bernmark}. Here we assume $K$ is a smooth, compact $m-$dimensional submanifold of $\R^d$ and $Q$ is a differentiable function of class $C^{1}$ on $K$; i.e., $Q$ is the restriction to $K$ of a differentiable function of class $C^{1}$ on $\R^d$. Fixing such $K$ and $Q$,  we let 
$$\mathcal P_n^Q:=\{f_n^Q(y):=\exp  {\bigl[-\sum_{j=2}^n\frac{1}{|y-x_j|^{\alpha}}-2nQ(y)\bigr]}: x_2,...,x_n\in K\}.$$

\begin{proposition} \label{derivest} There exist $C_n=C_n(K)>0$ with $\lim_{n\to \infty}C_n^{1/n}=1$ such that
$$||\nabla f_n^Q||_K\leq C_n ||f_n^Q||_K, \ \hbox{for all} \ f_n^Q\in \mathcal P_n^Q.$$

\end{proposition}

\begin{proof} Fix $f_n^Q\in \mathcal P_n^Q$. It suffices to estimate 
$\partial f_n^Q/\partial y_1$.
Since 
$$|y-x_j|^{\alpha}=\bigg(\sum_{k=1}^d [y_k - (x_j)_k]^2\bigg)^{\alpha/2},$$
we compute
$$ \frac{\partial f_n^Q}{\partial y_1}(y)= 
f_n^Q(y)\bigg(\alpha\sum_{j=2}^n \frac{y_1-(x_j)_1}{|y-x_j|^{\alpha+2}}-2n\frac{\partial Q}{\partial y_1}\bigg).$$
Let $B:=\sup_K |\nabla Q|$. Then on $K$, 
\begin{equation}
\label{grad} \left|\frac{\partial f_n^Q}{\partial y_1}(y)\right|\leq  
f_n^Q(y)\bigg(\sum_{j=2}^n \frac{\alpha}{|y-x_j|^{\alpha+1}}+2nB \bigg).
\end{equation}

We fix $\delta \geq 1$ and we break up the estimate of $|\partial f_n^Q/\partial y_1|$ into two cases:
\vskip4pt

\noindent{\it Case I}: $|y-x_j|\geq n^{-\delta}, \ j=2,...,n$.

\vskip4pt

In this case, we immediately obtain
$$\left|\frac{\partial f_n^Q}{\partial y_1}(y)\right|\leq 
\bigl(\alpha(n-1)\cdot n^{\delta(\alpha+1)}+2nB\bigr)  f_n^Q(y).$$

\vskip4pt

\noindent{\it Case II}: $|y-x_j|< n^{-\delta}$ for some $j\in \{2,...,n\}$.

\vskip4pt

We begin with a lower bound on $\|f_n^Q\|_{K}$ on $K$. Given $\beta >1$, we claim that for $n=n(\beta)$ sufficiently large, 
\begin{equation}\label{lower} ||f_n^Q||_K\geq e^{-n^{\alpha\beta+1}} e^{-2nC}\end{equation}
where $C=\max_KQ$. Since the volume of $(n-1)$ balls of radius $n^{-\beta}$ in $\R^{d}$ tends to zero as $n\to\infty$, we can find, for $n$ large, a point $y\in K$ with 
$$|y-x_j|\geq n^{-\beta} \ \hbox{for} \ j=2,...,n.$$
Thus 
$$-|y-x_j|^{-\alpha}\geq -n^{\alpha \beta}$$
so that
$$|f_n^Q(y)|\geq e^{-(n-1)n^{\alpha\beta}-2nC}\geq 
e^{-n^{\alpha\beta+1}} e^{-2nC}.$$

Suppose $|y-x_2|=\min_j |y-x_j|< n^{-\delta}$. Since $ \exp (-|y-x_j|^{-\alpha})\leq 1$, we have the trivial estimate
$$|f_n^Q(y)|\leq  \exp(-|y-x_2|^{-\alpha}-2nc),$$
where $c=\min_{K}Q$.
Inserting this in (\ref{grad}) and using $|y-x_j|\geq |y-x_2|$, we have
$$\left|\frac{\partial f_n^Q}{\partial y_1}(y)\right|\leq  
\exp(|y-x_2|^{-\alpha}-2nc)\bigg(\sum_{j=2}^n \frac{\alpha}{|y-x_j|^{\alpha+1}}+2nB \bigg)
$$
$$
\leq n e^{-2nc}\cdot \exp{\bigg(-\frac{1}{|y-x_2|^{\alpha}}\bigg)}\cdot \bigg(\frac{\alpha}{|y-x_2|^{\alpha}}+2B\bigg). $$
This upper bound is of the form $n e^{-2nc}g(v)$ where
$$g(v)= (\alpha v+2B)e^{-v}$$
and $v=|y-x_2|^{-\alpha}$. A calculation shows that $g$ is decreasing if $v>1$. 
Thus the maximum value of $|(\partial f_n^Q/\partial y_1)(y)|$ for 
$|y-x_2|\leq n^{-\delta}$ occurs 
when $|y-x_2|= n^{-\delta}$ which gives 
$$\left|\frac{\partial f_n^Q}{\partial y_1}(y)\right|\leq 
n e^{-2nc}(\alpha n^{\alpha \delta} +2B)e^{-n^{\alpha \delta}}.$$
From (\ref{lower}) we obtain
$$
\left|\frac{\partial f_n^Q}{\partial y_1}(y)\right|\leq \|f_n^Q\|_K e^{n^{\alpha\beta+1}} e^{2nC} n e^{-2nc}(\alpha n^{\alpha \delta} +2B)e^{-n^{\alpha \delta}}.$$
Choosing $\beta>1$ and $\delta>1$ so that $\alpha (\delta - \beta) >1$ gives the result.

\end{proof}

Now let $\mu$ be a positive measure on $K$ of finite total mass. Using the previous estimate, we can give a sufficient mass density condition for an appropriate Bernstein-Markov type property on the classes $\mathcal P_n^Q$ in the spirit of Proposition \ref{mdprop}. This can be used in proving probabilistic results involving equidistribution of random arrays of points in this setting.

\begin{proposition}\label{bernmark} Suppose there exist constants $r_0>0$ and $T>0$ such that 
$$\mu(B(x,r))\geq r^T$$
for all $x\in K$ and $r<r_0$. Then 
$$||f_n^Q||_K \leq M_n \int_K f_n^Q(x)d\mu(x) \ \hbox{for all} \ f_n^Q \in \mathcal P_n^Q$$
where $M_n^{1/n}\to 1$.
\end{proposition}

\begin{proof} Given $ f_n^Q \in \mathcal P_n^Q$, choose $p\in K$ with $||f_n^Q||_K=f_n^Q(p)$. Given $x\in K$ let $\gamma:[0,L]\to K$ be a smooth curve in $K$ parameterized by arclength joining $x=\gamma(0)$ to $p=\gamma(L)$. Then
$$f_n^Q(p)-f_n^Q(x)=\int_0^L \nabla f_n^Q(\gamma(t))\cdot \gamma'(t) dt.$$
Hence
$$f_n^Q(p)-f_n^Q(x)\leq L \cdot ||\nabla f_n^Q||_K.$$
From Proposition \ref{derivest}, 
$$f_n^Q(p)-f_n^Q(x)\leq L C_n ||f_n^Q||_K=LC_n f_n^Q(p)$$
so that 
$$f_n^Q(x)\geq f_n^Q(p)[1- C_nL].$$
Given $\epsilon >0$, we take $x\in B(p,e^{-\epsilon n})$. Since $K$ is a smooth submanifold of $\R^{d}$, for all such $x$, there is a constant $C$ depending on $K$ such that for $n$ sufficiently large, there exists a smooth curve $\gamma$ as above with $L\leq Ce^{-\epsilon n}$. Choosing such $x$ and $\gamma$, for $n$ large we have
$$f_n^Q(x) \geq \frac{1}{2} f_n^Q(p)$$
and $e^{-\epsilon n}<r_0$. Then
$$\int_K f_n^Q(x)d\mu(x)\geq \int_{B(p,e^{-\epsilon n})} f_n^Q(x)d\mu(x)\geq \frac{1}{2} f_n^Q(p) \mu (B(p,e^{-\epsilon n}))\geq \frac{1}{2}\|f_n^Q\|_K e^{-\epsilon n T}.$$

\end{proof}

\section{Bernstein-Markov on algebraic varieties}

Returning to pluripotential theory, if $K\subset \C^n$ is pluripolar, then $V_K^*\equiv +\infty$; indeed, $V_K=+\infty$ on $\C^n \setminus Z$ where $Z$ is pluripolar. This situation can still be interesting if $Z$ has some structure. In this vein we mention the following beautiful result of Sadullaev \cite{Sad}. 
\begin{theorem} Let $A$ be a pure $m-$dimensional, irreducible analytic subvariety of $\C^n$ where $1\leq m \leq n-1$. Then $A$ is algebraic if and only if for some (all) $K\subset A$ compact and nonpluripolar in $A$, $V_K$ in (\ref{veekpoly}) is locally bounded on $A$.

\end{theorem}

\noindent Here ``$K$ nonpluripolar in $A$'' means that $K\cap A^{reg}$ is nonpluripolar as a subset of the complex manifold $A^{reg}$ consisting of the regular points of the variety $A$. Note that $A$ -- and hence $K$ -- is pluripolar in $\C^n$ so $V_K^*\equiv \infty$; moreover, $V_K=\infty$ on $\C^n\setminus A$.

A trivial but illustrative example is to take $A=\{(z,w)\in \C^2: w=0\}$ so that $A$ can be identified with the complex $z-$plane in $\C^2$. For any compact set $K\subset A$, by the definition of $V_K$ in (\ref{veekpoly}) we have 
$V_K|_A$ is the univariate extremal function and hence $V_K|_A$ is locally bounded on $A$ precisely when $K$ is polar in $A$. Because of Sadullaev's theorem, for $K\subset A$ compact and nonpluripolar in an algebraic subvariety of $\C^n$, one can define and discuss polynomial inequalities and Bernstein-Markov properties in this setting; one simply considers the polynomials restricted to $A$. In \cite{P2}, a sufficient mass density condition on a measure supported on a nonpluripolar subset $K$ of an algebraic variety $A\subset \C^n$ to satisfy a Bernstein-Markov property in the spirit of Theorem \ref{suffmd} is proved.

An example where the situation of this section arises can be found in \cite{LPLD}. Let $T$ denote the inverse stereographic projection $T:\C \cup \{\infty\}\to {\bf S}$ where ${\bf S}$ is the sphere in $\R^3$ centered in $(0,0,1/2)$ of radius $1/2$:
\begin{equation} \label{stereo} T(z)=v:=(v_1,v_2,v_3)=\left(\frac{\Re(z)}{1+|z|^{2}},\frac{\Im(z)}{1+|z|^{2}},\frac{|z|^{2}}{1+|z|^{2}}\right),\quad z\in\C\end{equation}
and $T(\infty)=P_0,$ where $P_0=(0,0,1)\in {\bf S}$. The set ${\bf S}$ itself is a nonpluripolar subset of the algebraic variety
$$A=\{(w_1,w_2,w_3)\in \C^3: w_1^2 +w_2^2+(w_3-1/2)^2=1\}\subset \C^3.$$
If $K\subset \C$ is a closed and {\it unbounded} set, then $T(K)$ is a compact subset of ${\bf S}$; and polynomials $p\in \mathcal P_k$ in $\C$ correspond to functions of the form $q(v)=\prod_{j=1}^{k}|v-a_j|$, 
$v, a_j\in {\bf S}$. Thus, to verify Bernstein-Markov properties for polynomials on certain unbounded sets in $\C$, such as $\R$, one can translate the problem to a compact setting on ${\bf S}$.

\section{(Pluri-)subharmonic Bernstein-Markov property}

Recall from section 3 that $\mu$ is a {\it determining measure} for $K$ if for all carriers $E$ of $\mu$, 
$V_E^*=V_K^*$. In \cite{BBN}, Berman, Boucksom and Nystrom give an interesting characterization of  determining measures $\mu$ for $K$ which satisfy the additional hypothesis that $\mu(Z)=0$ if $Z$ is pluripolar.

\begin{proposition} \label{bbn} Let $K\subset\C^n$ be compact and regular. Let $\mu$ be a probability measure on $K$ which puts no mass on pluripolar sets. Then $\mu$ is determining for $K$ if and only if $(K,\mu)$ satisfies a Bernstein-Markov property for psh functions; i.e., for all $\epsilon >0$, there exists $C=C(\epsilon,K)$ such that for all $p\geq 1$,
\begin{equation} \label{pshbm} \sup_K e^{pu}=(\sup_K e^{u})^p \leq C(1+\epsilon)^p ||e^u||^p_{L^p(\mu)}\end{equation}
for all $u\in L(\C^n)$. 
\end{proposition}

We remark that $\mu$ is determining for $K$ is equivalent to: {\it for each $u\in L(\C^n)$, $u\leq 0$ $\mu-$a.e. implies $u\leq 0$ on $K$} which is equivalent to: {\it for each $u\in L(\C^n)$, $\sup_K e^u =  ||e^u||_{L^{\infty}(\mu)}$}. Thus if $\mu$ is determining for $K$, (\ref{pshbm}) is equivalent to
$$||e^{u}||^p_{L^{\infty}(\mu)} \leq C(1+\epsilon)^p ||e^u||^p_{L^p(\mu)}. $$
Also, one does not need to require $u\in L(\C^n)$ as $u$ psh in a neighborhood of $K$ suffices; but to emphasize the connection with a (polynomial) Bernstein-Markov property, we use $u\in L(\C^n)$. Indeed, since $\frac{1}{deg (p)}\log|p|\in L(\C^n)$, an elementary argument shows this psh Bernstein-Markov property (\ref{pshbm}) implies a (polynomial) Bernstein-Markov property (compare (\ref{used})).

Berman, Boucksom and Nystrom raised the following question: 
\\[.5\baselineskip]
{\sl For a regular compact set $K\subset\C^n$ and $\mu$ a probability measure with support $K$ which puts no mass on pluripolar sets, is it true that if $(K,\mu)$ satisfies a Bernstein-Markov property {\it for polynomials} then $(K,\mu)$ satisfies a Bernstein-Markov property for {\it plurisubharmonic} functions (and hence these notions are equivalent)?} 
\\[.5\baselineskip]
Note the measure $\nu$ constructed on a general compact set $K$ in Proposition \ref{universal} is carried by a countable (and hence pluripolar) set. For a regular set $K$, by ``spreading out the mass'' from points to small nonpluripolar sets one is able to construct examples illustrating the answer to the question is negative.

\begin{example} We present an explicit example, due to V. Totik \cite{T}, of a positive measure $\mu$ with support equal to the interval $K=[0,1]$ such that

\begin{enumerate}
\item $\mu$ does not charge polar sets;
\item $(K,\mu)$ satisfy the {\it polynomial} Bernstein-Markov property; and  
\item $\mu$ is not a determining measure for $K$.
\end{enumerate}

\noindent Proposition \ref{bbn} then shows that the polynomial Bernstein-Markov property is not equivalent to the subharmonic Bernstein-Markov property.

The idea is to take a standard construction of a discrete measure on $K=[0,1]$ which satisfies the polynomial Bernstein-Markov property and then spread out the mass from points to sufficiently small intervals. The measure $\mu$ will be of the form $\mu=\sum_{n=1}^{\infty} \mu_n$ where for each $n=1,2,...$ we will choose $\epsilon_n >0$ in a specified way and, letting $dx$ be Lebesgue measure on $\R$,
$$\frac{d\mu_n}{dx} := {1\over n^5} \sum_{k=1}^{n^3-1}{1\over \epsilon_n} \cdot \chi_{[k/n^3 -\epsilon_n,k/n^3 +\epsilon_n]}$$
($\chi_A$ is the characteristic function of $A$). Then $\mu_n(K) \leq 2/n^2$ so $\mu$ is a finite measure and clearly the support of $\mu$ is $K=[0,1]$. The sequence $\{\epsilon_n\}$ is chosen so that the capacity 
$$C\bigl(\cup_{n=1}^{\infty} \cup_{k=1}^{n^3-1}[k/n^3 -\epsilon_n,k/n^3 +\epsilon_n]\bigr) <1/8$$
(e.g., apply Theorem 5.1.4 a) of \cite{R} with $d=1$). Clearly $\mu$ is not determining for $K$ since it has carriers $E$ with $C(E)<1/8$. Moreover, if $F$ is polar, $\mu_n(F)=0$ (since $\mu_n$ is absolutely continuous) and hence $\mu(F)=0$. 

It remains to show that $(K,\mu)$ satisfies the polynomial  Bernstein-Markov property. This is straightforward: if $p_n$ is a polynomial of degree $n$, let $||p_n||_K = M=|p_n(x_0)|$. By the classical Markov inequality, $||p_n'||_K\leq Mn^2$. Choose $k$ with $|k/n^3 - x_0|\leq 1/n^3$. Then for $x$ in the interval $[k/n^3 -\epsilon_n,k/n^3 +\epsilon_n]$, we have (if $n>4$)
$$|p_n(x)|\geq M - Mn^2\cdot 2/n^3 > M/2.$$
Hence
$$\int_K |p_n|^2 d\mu \geq  \int_K |p_n|^2 d\mu_n \geq (M/2)^2\cdot 1/n^5 \cdot (2\epsilon_n)\cdot 1/\epsilon_n$$
$$\geq (M/2)^2\cdot 1/n^5 = {1\over 4n^5} ||p_n||_K^2.$$\end{example}

In \cite{BBN}, given $K\subset \C^n$ compact and $Q$ continuous on $K$, one defines a measure $\mu$ with support in $K$ to be {\it determining for $(K,Q)$} if for all carriers $E$ of $\mu$ we have $V^*_{E,Q}=V^*_{K,Q}$. Here for such $E$ we define
$$V_{E,Q}(z):=\sup \{u(z): \ u\in L(\C^n), \ u\leq Q \ \hbox{on} \ E\}$$ as in (\ref{vkq}). Using this definition, a weighted version of Proposition \ref{bbn} if $V_{K,Q}$ is continuous is proved. In particular, in this case, $\mu_{K,Q}= (dd^cV_{K,Q})^n$ is determining for $(K,Q)$ and the triple  $(K,\mu_{K,Q},Q)$ satisfies a weighted plurisubharmonic Bernstein-Markov property and hence a weighted polynomial Bernstein-Markov property.
  
\section{Final remarks and open questions}

In $\C$, Proposition \ref{znunwtd} states that if $(K,\nu)$ satisfy a Bernstein-Markov inequality then
\begin{equation}\label{zeen} 
\lim_{k\to \infty} Z_k^{1/k^2} = \delta(K).
\end{equation}
A slightly weaker condition than the Bernstein-Markov property is sufficient for (\ref{zeen}). We say $\mu \in {\bf Reg}$ if the sequence of monic orthogonal polynomials $p_k(z)=z^k + \cdots$ for $\mu$ satisfy 
\begin{equation}\label{zeen2}\lim_{k\to \infty} ||p_k||_{L^2(\mu)}^{1/k} = \delta (K). \end{equation}
This is equivalent to: for any sequence $\{p_j\}$ of nonzero polynomials with $deg(p_j)\to +\infty$,
\begin{equation}\label{zeen3}
\limsup_{j\to \infty} \bigg(\frac{|p_j(z)|}{||p_j||_{L^2(\mu)}}\bigg)^{1/deg (p_j)} \leq \exp(V_K^*(z)), \ z\in \C. 
\end{equation}
For $n\geq 1$ if $K\subset \C^n$ is regular then (\ref{zeen3}) for $z\in \C^n$ is 
equivalent to the Bernstein-Markov property (\ref{bmk}). Bloom \cite{bloomreg} has recently shown that for $K\subset \C$,  (\ref{zeen2}) is equivalent to (\ref{zeen}). In $\C^n$ for $n>1$, one can take (\ref{zeen3}) as the definition of $\mu \in {\bf Reg}$ (cf., \cite{Biumj}).

Suppose that $K\subset \C^n$ is compact. Given a measure $\nu$ with support in $K$ and $\{q_j^{(k)}\}_{j=1,...,N_k}$ an orthonormal basis for $\mathcal P_k$ with respect to $L^2(\nu)$, 
$$
B_k^{\nu}(z):=\sum_{j=1}^{N_k} |q_j^{(k)}(z)|^2
$$
is the {\it $k-th$ Bergman function of $K,\nu$}. If $K=\{z\in \C:|z|\leq 1\}$ and $\nu=\frac{1}{2\pi}d\theta = \mu_K$, it was observed in the introduction that the monomials $1,z,...,z^k$ give an orthonormal basis for $\mathcal P_k$ in $L^2(\nu)$. Hence 
$$B_k^{\nu}(z)=\sum_{j=0}^k |z|^{2j}= \frac{|z|^{2k+2}-1}{|z|^2-1}$$
and it follows easily that
$$ \lim_{k\to \infty}\frac{1}{2k} \log B_k^{\nu}(z)=\log^+|z|=V_K(z)$$ 
locally uniformly in $\C$ (recall (\ref{veek})). More generally, if $V_K$ is continuous and $(K,\nu)$ satisfies a Bernstein-Markov property then
\begin{equation}
\label{unwtdweakasym}
\lim_{k\to \infty} \frac{1}{2k} \log B_k^{\nu}(z) = V_{K}(z)
\end{equation}
locally uniformly on $\C^n$ (cf., \cite{bloomshiff}). 

In the weighted setting, given a measure $\nu$ and a weight $Q$ on $K$, and letting $\{q_j^{(k)}\}_{j=1,...,N_k}$ be an orthonormal basis for $\mathcal P_k$ with respect to the weighted $L^2-$norm $p_k\to ||w^kp_k||_{L^2(\nu)}$, we call
$$B_k^{\nu,w}(z):=\sum_{j=1}^{N_k} |q_j^{(k)}(z)|^2\cdot w(z)^{2k}$$
the $k-th$ {\it Bergman function} of $K,Q,\nu$. If $V_{K,Q}$ is continuous and $(K,\nu,Q)$ satisfies a weighted Bernstein-Markov property we have 
\begin{equation}
\label{weakasym}
\lim_{k\to \infty} \frac{1}{2k} \log B_k^{\nu,w}(z) = V_{K,Q}(z)
\end{equation}
locally uniformly on $\C^n$ (cf., \cite{bloom}). The local uniform convergence in (\ref{weakasym}) implies weak-* convergence of the Monge-Amp\`ere measures
$$[dd^c\frac{1}{2k} \log B_k^{\nu,w}(z)]^n \to (dd^cV_{K,Q}^*)^n.$$
The asymptotic relations (\ref{unwtdweakasym}) and (\ref{weakasym}) form the basis for probabilistic results on zeroes of random polynomials and random polynomial mappings; cf., \cite{bayrak}, \cite{bloom}, \cite{bloom2}, \cite{bloomshiff}, and \cite{blrp}. 

The growth of the Bergman functions $\{B_k^{\nu}\}$ and $\{B_k^{\nu,w}\}$ detects (weighted) Bernstein-Markov properties of $\nu$. Indeed, given $p_k\in \mathcal P_k$, we may write $p_k(z)=\sum_{j=0}^k c_j q_j^{(k)}(z)$ where $c_j=\int_K p_k \overline{q_j^{(k)}}w^{2k}d\nu$ and we have 
$$ |p_k(z)|\leq \sqrt {B_k^{\nu,w}(z)}\cdot ||p_k||_{L^2(\nu)}.$$
Thus we see the following:
\begin{enumerate}
\item If $\bigl(\sup_K B_k^{\nu,w}(z)\bigr)^{1/2k}\to 1$, then $(K,\nu,Q)$ satisfy a weighted Bernstein-Markov property; hence
\item if $V_{K,Q}$ is continuous and (\ref{weakasym}) holds uniformly on $K$, then $(K,\nu,Q)$ satisfy a weighted Bernstein-Markov property
\end{enumerate}
\noindent with similar remarks holding in the unweighted case; e.g., if $V_K$ is continuous and (\ref{unwtdweakasym}) holds uniformly on $K$, then $(K,\nu)$ satisfy a Bernstein-Markov property. 

If $V_{K,Q}$ is continuous and $(K,\nu,Q)$ satisfy a weighted Bernstein-Markov property, it is proved in \cite{BBnew} that the sequence $\{\frac{1}{N_k}B_k^{\nu,w} d\nu\}$ of probability measures on $K$ satisfies
\begin{equation}
\label{strongasym}\frac{1}{N_k}B_k^{\nu,w} d\nu \to \mu_{K,Q}:=(dd^cV_{K,Q}^*)^n \ \hbox{weak-}*.\end{equation}
For $n=1$ this result is in \cite{blchrist}. The unweighted version of (\ref{strongasym}) is
\begin{equation}
\label{strongasymun}\frac{1}{N_k}B_k^{\nu} d\nu \to \mu_{K}:=(dd^cV_{K}^*)^n \ \hbox{weak-}*.\end{equation}

\medskip

We end with some open questions.

\begin{enumerate}

\item {\sl Does there exist a measure $\mu$ with support $K$ such that $K$ is regular and $(K,\mu)$ satisfy a Bernstein-Markov property but $\mu$ does not satisfy (\ref{md??}) for any $T>0$?}

\item {\sl (Erd\"os; cf., section 6.7 of \cite{ST}) For $\mu$ with support $K = [-1,1]$ and $d\mu(x) = w(x)dx$ with $w:[-1,1]\to [0,\infty)$ bounded, is it true that $\mu \in {\bf Reg}$ implies
$$\lim_{\epsilon \to 0} C(E_{\epsilon})=C(K)$$
where $E_{\epsilon}$ is any set obtained from $\{x\in K: w(x) >0\}$ by removing a
subset of measure at most $\epsilon$? The converse statement (due to Ullman and Erd\"os-Freud) follows from Theorem 6.7.2 of \cite{ST}.}

\item {\sl Is there an analogue of the italicized statement in the proof of Theorem \ref{suffmd} for $W_K$ in (\ref{wk}); i.e., given $\epsilon >0$, does there exist $\delta >0$ so that for any compact subset $E$ of $K$ with $C(E)>C(K) -\delta$ we have $W_E^*\leq W_K^* +\epsilon$ (perhaps with a different notion of capacity)?} 

\item {\sl In $\C^n$ for $n>1$, if $K$ is regular, is (\ref{bmk}) equivalent to (\ref{zeen})?}

\item {\sl Is (\ref{zeen3}) equivalent to (\ref{zeen}) in $\C^n$ for $n>1$?}

\item {\sl Does (\ref{strongasymun}) imply (\ref{unwtdweakasym})?}

\end{enumerate}

{\bf Authors:}\\[\baselineskip]
T. Bloom, bloom@math.toronto.edu\\
University of Toronto, Toronto, CA\\
\\[\baselineskip]
N. Levenberg, nlevenbe@indiana.edu\\
Indiana University, Bloomington, IN, USA\\
\\[\baselineskip]
F. Piazzon, fpiazzon@math.unipd.it\\
University of Padua, Padua, ITALY\\
\\[\baselineskip]
F. Wielonsky, franck.wielonsky@univ-amu.fr\\
Universit\'e Aix-Marseille, Marseille, FRANCE\\
\\[\baselineskip]

\end{document}